\documentclass[12pt]{amsart}
\usepackage{amsbsy}
\usepackage{graphicx,epsfig,subfigure,psfrag}
\begin{document}

\newtheorem{theorem}{Theorem}
\newtheorem{proposition}{Proposition}
\newtheorem{lemma}{Lemma}
\newtheorem{corollary}{Corollary}
\newtheorem{definition}{Definition}
\newtheorem{remark}{Remark}
\newcommand{\tex}{\textstyle}
\numberwithin{equation}{section} \numberwithin{theorem}{section}
\numberwithin{proposition}{section} \numberwithin{lemma}{section}
\numberwithin{corollary}{section}
\numberwithin{definition}{section} \numberwithin{remark}{section}
\newenvironment{proofof}[1][Proof of Lemma \ref{Le 42}]{\textbf{#1.} }{\ \rule{0.5em}{0.5em}}
\newcommand{\ren}{\mathbb{R}^N}
\newcommand{\re}{\mathbb{R}}
\newcommand{\n}{\nabla}
\newcommand{\p}{\partial}
\newcommand{\iy}{\infty}
\newcommand{\pa}{\partial}
\newcommand{\fp}{\noindent}
\newcommand{\ms}{\medskip\vskip-.1cm}
\newcommand{\mpb}{\medskip}
\newcommand{\AAA}{{\bf A}}
\newcommand{\BB}{{\bf B}}
\newcommand{\CC}{{\bf C}}
\newcommand{\DD}{{\bf D}}
\newcommand{\EE}{{\bf E}}
\newcommand{\FF}{{\bf F}}
\newcommand{\GG}{{\bf G}}
\newcommand{\oo}{{\mathbf \omega}}
\newcommand{\Am}{{\bf A}_{2m}}
\newcommand{\CCC}{{\mathbf  C}}
\newcommand{\II}{{\mathrm{Im}}\,}
\newcommand{\RR}{{\mathrm{Re}}\,}
\newcommand{\eee}{{\mathrm  e}}
\newcommand{\LL}{L^2_\rho(\ren)}
\newcommand{\LLL}{L^2_{\rho^*}(\ren)}
\renewcommand{\a}{\alpha}
\renewcommand{\b}{\beta}
\newcommand{\g}{\gamma}
\newcommand{\G}{\Gamma}
\renewcommand{\d}{\delta}
\newcommand{\D}{\Delta}
\newcommand{\e}{\varepsilon}
\newcommand{\var}{\varphi}
\newcommand{\lll}{\l}
\renewcommand{\l}{\lambda}
\renewcommand{\o}{\omega}
\renewcommand{\O}{\Omega}
\newcommand{\s}{\sigma}
\renewcommand{\t}{\tau}
\renewcommand{\th}{\theta}
\newcommand{\z}{\zeta}
\newcommand{\wx}{\widetilde x}
\newcommand{\wt}{\widetilde t}
\newcommand{\noi}{\noindent}
\newcommand{\uu}{{\bf u}}
\newcommand{\xx}{{\bf x}}
\newcommand{\yy}{{\bf y}}
\newcommand{\zz}{{\bf z}}
\newcommand{\aaa}{{\bf a}}
\newcommand{\cc}{{\bf c}}
\newcommand{\jj}{{\bf j}}
\newcommand{\ggg}{{\bf g}}
\newcommand{\UU}{{\bf U}}
\newcommand{\YY}{{\bf Y}}
\newcommand{\HH}{{\bf H}}
\newcommand{\GGG}{{\bf G}}
\newcommand{\VV}{{\bf V}}
\newcommand{\ww}{{\bf w}}
\newcommand{\vv}{{\bf v}}
\newcommand{\hh}{{\bf h}}
\newcommand{\di}{{\rm div}\,}
\newcommand{\ii}{{\rm i}\,}
\newcommand{\inA}{\quad \mbox{in} \quad \ren \times \re_+}
\newcommand{\inB}{\quad \mbox{in} \quad}
\newcommand{\inC}{\quad \mbox{in} \quad \re \times \re_+}
\newcommand{\inD}{\quad \mbox{in} \quad \re}
\newcommand{\forA}{\quad \mbox{for} \quad}
\newcommand{\whereA}{,\quad \mbox{where} \quad}
\newcommand{\asA}{\quad \mbox{as} \quad}
\newcommand{\andA}{\quad \mbox{and} \quad}
\newcommand{\withA}{,\quad \mbox{with} \quad}
\newcommand{\orA}{,\quad \mbox{or} \quad}
\newcommand{\atA}{\quad \mbox{at} \quad}
\newcommand{\onA}{\quad \mbox{on} \quad}
\newcommand{\ef}{\eqref}
\newcommand{\mc}{\mathcal}
\newcommand{\mf}{\mathfrak}

\newcommand{\ssk}{\smallskip}
\newcommand{\LongA}{\quad \Longrightarrow \quad}
\def\com#1{\fbox{\parbox{6in}{\texttt{#1}}}}
\def\N{{\mathbb N}}
\def\A{{\cal A}}
\newcommand{\de}{\,d}
\newcommand{\eps}{\varepsilon}
\newcommand{\be}{\begin{equation}}
\newcommand{\ee}{\end{equation}}
\newcommand{\spt}{{\mbox spt}}
\newcommand{\ind}{{\mbox ind}}
\newcommand{\supp}{{\mbox supp}}
\newcommand{\dip}{\displaystyle}
\newcommand{\prt}{\partial}
\renewcommand{\theequation}{\thesection.\arabic{equation}}
\renewcommand{\baselinestretch}{1.1}
\newcommand{\Dm}{(-\D)^m}

\title
{\bf Variational approach for a class of cooperative systems}

\author{P.~\'Alvarez-Caudevilla}

\address{Universidad Carlos III de Madrid,
Av. Universidad 30, 28911-Legan\'es, Spain -- Work phone number:
+34-916249099} \email{alvcau.pablo@gmail.com}

\keywords{Coexistence
states. Cooperative systems. Variational Methods. Non--local problems.}

\thanks{The author is partially supported by the Ministry of Science and Innovation of
Spain under the grant MTM2009-08259.}

 \subjclass{35K40, 35K50, 35K57, 35K65. }
\date{\today}





\begin{abstract}

The aim of this 
work is to ascertain the characterization of the 
existence of coexistence states for a class of cooperative systems 
supported by the study of an associated non--local 
equation through classical variational methods. Thanks to those results we are able 
to obtain the blow--up behaviour of the solutions in the whole domain for 
certain values of the main continuation parameter.

\end{abstract}

\maketitle

\section{Introduction}


\subsection{Model, spatial distribution and notation}


\noindent We consider the following cooperative elliptic system
\begin{equation}
\label{11}
  \left\{ \begin{array}{ll} \begin{array}{l}
  -\Delta u=\l u + \a v - a f(x,u) u\\
  -\Delta v=\b u + \l v  \end{array}
  & \quad \hbox{in }\;\; \O,  \\ (u,v)=(0,0)
  &  \quad \hbox{on }\;\;\p \O.\\
  \end{array}\right.
\end{equation}
where $\O$ is a bounded domain of $\re^N$,
$N\geq 1$, with boundary $\p \O$ of class $\mc{C}^{2+\mu}$ for some
$\mu\in(0,1)$, $\l\in\re$ $\a>0$ and $\b>0$ are regarded as real 
continuation parameters, $\D$ stands
for the Laplacian operator in $\re^N$, and $a\in \mc{C}^\mu(\bar
\O)$ is a non-negative function satisfying the following
hypothesis, which will be maintained throughout this work:
\begin{enumerate}
    \item[A.] The open set
    \begin{equation*}
        \O_+:= \{\,x\in\O\;\;:\;a(x)>0\,\},
    \end{equation*}
    is a subdomain of $\O$ of class $\mathcal{C}^{2+\mu}$
    with $\bar \O_+\subset \O$.
    \item[B.] The open set
    \begin{equation*}
    \O_0:=\O\setminus \bar \O_+,
    \end{equation*}
    is a subdomain of $\O$ of class $\mathcal{C}^{2+\mu}$ 
    such that
    \begin{equation*}
        K_0:= (a)^{-1}(0) = \bar \O \setminus \Omega_+,
    \end{equation*}
    is a compact set. Moreover, $\p\O_0$ consists 
    of two components, $\G_1$ and $\G_2$ and are also of
		class $\mc{C}^{2+\mu}$. 
\end{enumerate}
Figure \ref{CoopVarFig-01.pdf} shows a typical situation where the conditions 
A and B are fulfilled.
\begin{figure}[ht]
\begin{center}
\includegraphics[height=8cm]{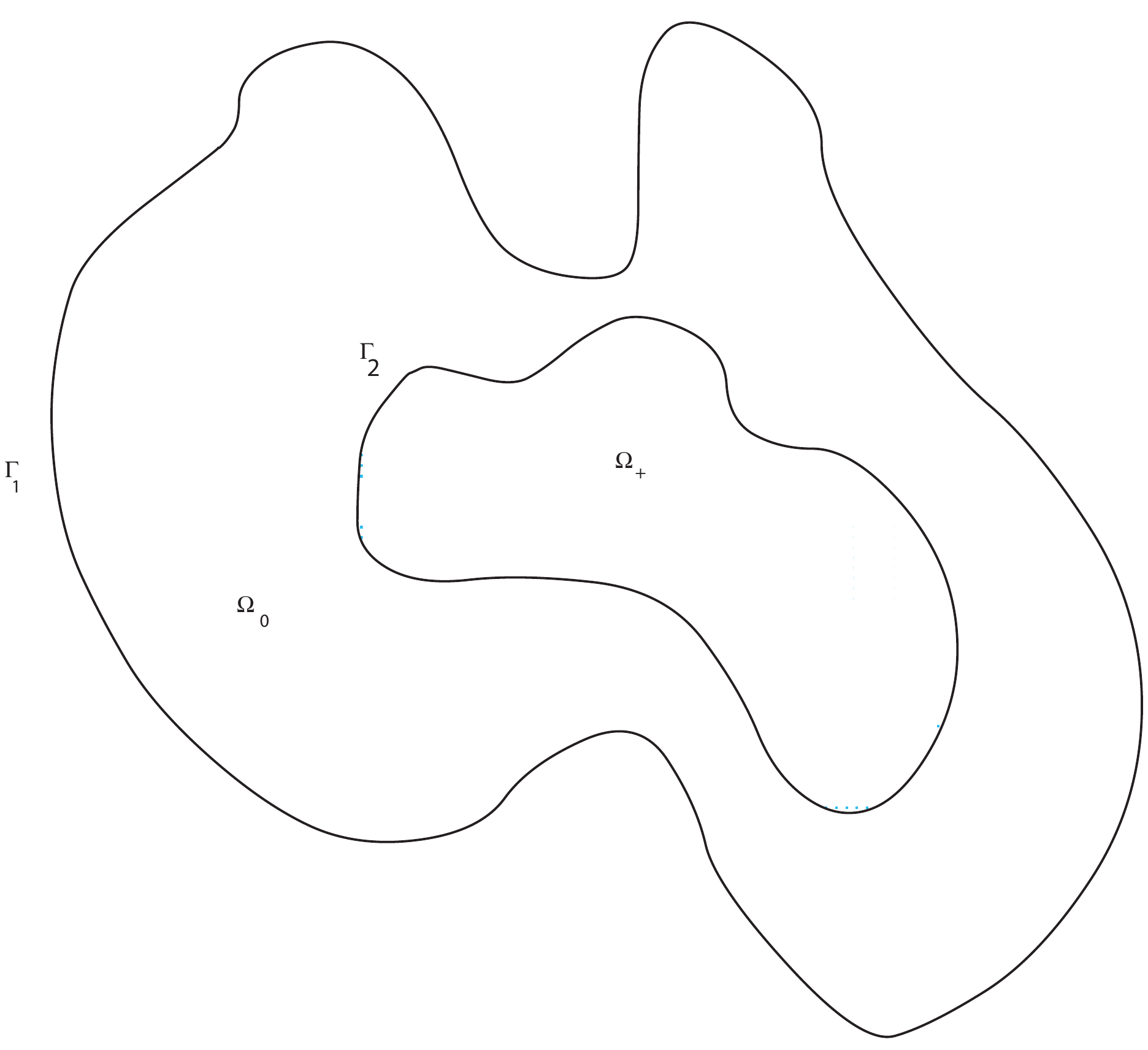}
\caption{Nodal configuration of $a(x)$.} \label{CoopVarFig-01.pdf}
\end{center}
\end{figure} 
Furthermore, for the function $f(x,u)$, we suppose the following assumptions:
\begin{enumerate}
\item[(Af)] $f \in \mathcal{C}^{\mu,1+\mu}(\bar\O\times
[0,\infty))$ satisfies
\begin{equation*}
f(x,0)=0\quad \hbox{and}\quad \p_u f(x,u)>0\quad\hbox{for
all}\quad x \in\bar\O\quad \hbox{and}\quad u>0\,.
\end{equation*}
\item[(Ag)] There exists $g \in \mathcal{C}^{1+\mu}([0,\infty))$
such that
\begin{equation*}
  g(0)=0\,, \;\; g(u)>0\;\;\hbox{and}\;\; g'(u)>0 \;\;
  \hbox{for all} \;\; u >0 \,,\quad
  \lim_{u\to\infty}g(u)=\infty\,,
\end{equation*}
where $'=\frac{d}{du}$, and
\begin{equation*}
  f(\cdot,u)\geq g(u)\quad \hbox{if}\quad u\geq 0\,.
\end{equation*}
\end{enumerate}
Note that (Af), (Ag) imply
\begin{equation*}
  \lim_{u\uparrow \infty}f(x,u)=\infty \quad \hbox{uniformly in}\quad
  x\in \bar \O\,.
\end{equation*}
Therefore, under these circumstances we are supposing that the nonlinearities 
of the two equations involved in \eqref{11} vanish in different subdomains of 
$\O$. Indeed, the nonlinearity of the v-equation 
vanishes overall $\O$ so it will be a linear equation, 
while the other is semilinear degenerated whose nonlinearity  
vanishes only in $\O_0$. Hence, this is a very novel 
situation, different from these usually analyzed in the literature. 
In particular, 
in the works of Molina-Meyer \cite{Molina1}, \cite{Molina2}, \cite{Molina3} and 
\cite{AL-NA} a class of cooperative systems, assuming a situation where 
the nonlinearities vanish in the same subdomains,  
such characterization of the 
existence of coexistence states was obtained 
in terms of the parameter $\l$.

The spatial distribution imposed throughout this paper was first analyzed 
in \cite{AL4}, but using the method of sub and supersolutions. On the contrary, here we base our analysis on 
variational methods \cite{D4,De,R2} which allow us to answer some of the open questions which arose in \cite{AL4}. 
Furthermore, the analysis carried out here might be crucial as a 
step forward in ascertaining the dynamics of more general classes of 
cooperative parabolic problems with general non-negative coefficients in front of the nonlinearities.
\par
Next, we introduce some notations. Then, for every
$V_1,V_2 \in \mathcal{C}^\mu(\bar \O)$ we denote by 
\begin{equation}
\label{12}
     \mf{L}(V_1,V_2) :=\left(\begin{array}{cc} -\D +V_1 & -\a\\
     -\b & -\D+V_2 \end{array}\right).
\end{equation} 
the \emph{strongly cooperative operator} (as
discussed in \cite{LM}, \cite{Am3}, \cite{FM}, \cite{Sw}) in the sense that
\begin{equation}
\label{13}
 	\a>0 \quad \hbox{and}\quad \b>0
 	\quad \hbox{in}\;\;\bar \O.
\end{equation} 
Therefore, 
for any smooth $D \subset \O$, there is a
unique value $\tau$ for which the linear eigenvalue problem 
\begin{equation}
\label{14}
    \left\{ \begin{array}{ll}
    \mf{L}(V_1,V_2) \binom {\varphi}{\psi} =
    \tau \binom {\varphi}{\psi} & \quad \hbox{in}\quad D, \\
    (\varphi,\psi)= (0,0) & \quad \hbox{on}\quad \p
    D,\end{array}\right.
\end{equation}
possesses a solution $(\varphi,\psi)$ with $\varphi
>0$ and $\psi >0$. Thus, we denote by $\s[\mf{L}(V_1,V_2),D]$ 
the \emph{principal eigenvalue} of $\mf{L}(V_1,V_2)$ in $D$ (under
homogeneous Dirichlet boundary conditions) and it is well known 
that $\s[\mf{L}(V_1,V_2),D]$ is \emph{simple} and
\emph{dominant}, in the sense that
\begin{equation*}
  \mathrm{Re\,} \tau > \s[\mf{L}(V_1,V_2),D],
\end{equation*}
for any other eigenvalue $\tau$ of \eqref{14}. Moreover, the
\emph{principal eigenfunction} $(\varphi,\psi)$ is unique, up to a
positive multiplicative constant, and
\begin{equation*}
  \varphi \gg 0, \qquad \psi\gg 0.
\end{equation*}
A function $w\in \mc{C}^1(\bar D)$ is said to satisfy $w\gg 0$ (in
$D$) if it lies in the interior of the cone of non-negative
functions of $\mc{C}^1(\bar D)$, i.e., if $w(x)>0$ for all $x\in D$
and $\p w /\p n (x)<0$ for all $x \in w^{-1}(0)\cap \p D$, where
$n=n(x)$ stands for the outward unit normal to $D$ at $x\in\p D$.
\par
Also, throughout this paper we set
\begin{equation}
\label{15}
  \mathfrak{L}_0:=\mathfrak{L}(0,0)\, \qquad
  \sigma_1:=\sigma[-\Delta,\Omega],
\end{equation}
and denote by $\phi_1\gg 0$ the principal eigenfunction associated
with $\sigma_1$, normalized so that, for example,
\begin{equation*}
  \max_{\bar \Omega} \phi_1=1.
\end{equation*}
Hence, performing some calculations and 
according to the properties of the principal eigenvalue we arrive at 
\begin{equation}
\label{16}
  \sigma[\mathfrak{L}_0,\Omega]= \sigma_1-\sqrt{\alpha\beta}
  \qquad \hbox{and} \qquad (\varphi,\psi)=\left( \sqrt{\alpha}\;
  \phi_1,\sqrt{\beta}\; \phi_1\right).
\end{equation}
Similarly, we also obtain
\begin{equation*}
   \sigma[\mathfrak{L}_0,\Omega_0]=\sigma_1^0-\sqrt{\alpha\beta},
\end{equation*}
where
\begin{equation*}
  \sigma_1^0:=\sigma[-\Delta,\Omega_0].
\end{equation*}

\subsection{Motivations and main results} 


\noindent Due to the cooperative character of system \eqref{11} 
we apply \cite[Lemma\;3.6]{AL-NA} which guarantees that   
$v=0$ if $u=0$ from the $u$-equation 
because $\alpha>0$ in $\bar\O$. Similarly, since $\beta>0$ in $\bar\O$, 
it follows from the
$v$-equation that $u=0$ if $v=0$. Thus, \eqref{11} admits two types
of non-negative solutions: the \emph{trivial state} $(0,0)$, and the
\emph{coexistence states}; those of the form $(u,v)$ with $u\gg 0$
and $v\gg 0$. Moreover, according to the $v$-equation, if $(u,v)$ is
a coexistence state of \eqref{11}, then
\begin{equation*}
  (-\Delta -\lambda)v = \beta u >0\qquad \hbox{in}\;\; \Omega,
\end{equation*} 
Therefore, owing to \cite[Theorem 2.1]{LM} (cf. \cite[Theorem 2.5]{Lo3}), under
homogeneous Dirichlet boundary conditions,  
the following condition must be held for the existence of 
coexistence states
\begin{equation}
\label{17}
  \lambda < \sigma_1.
\end{equation}
Then, if condition \eqref{17} is satisfied we can solve the v-equation 
in terms of u and under homogeneous Dirichlet boundary conditions 
$v=0$ on $\p\O$, i.e.,
\begin{equation*}
  v = \beta (-\Delta-\lambda)^{-1}u,
\end{equation*}
with $(-\Delta-\lambda)^{-1}$ as a positive linear integral compact 
operator from $L^2(\O)$ to itself.
Substituting that expression into the first equation of \eqref{11}, 
we obtain the following non--local problem for u:
\begin{equation}
\label{18}
  \left\{ \begin{array}{ll}
  -\Delta u  = \lambda u + \alpha\beta (-\Delta-\lambda)^{-1} u
  - a(x) f(x,u)u & \quad \hbox{in} \;\; \Omega,
  \\ u =0 & \quad \hbox{on} \;\; \partial\Omega,\end{array} \right.
\end{equation}
This integro-differential equation can be viewed as the Euler--Lagrange 
equation of the functional
\begin{equation}
\label{19}
	\mc{E}_\g (u) =\int_\O \left[\frac{1}{2}|\nabla u|^2 - 
	\frac{\l}{2}u^2- \frac{\a\b}{2} 
	\left|(-\Delta-\lambda)^{-1/2} u\right|^{2}+
	 F(x,u)\right], 
\end{equation}
where 
\begin{equation*}
  F(x,u):=\int_0^{u} a(x) f(x,\xi) \xi d\xi \quad \hbox{and}\quad \g:=\a \b.
\end{equation*}
Note that owing to the definition of $(-\Delta-\lambda)^{-1}$ 
the operator $(-\Delta-\lambda)^{-1/2}$ is defined as the square root of 
$(-\Delta-\lambda)^{-1}$ and it will be also referred to as a non-local 
linear operator. Furthermore, due to the relation with \eqref{18} the analysis carried out here might add some 
additional valuable information and new methods to the analysis 
of higher-order partial differential equations (cf. \cite{PV}). 
\par
After these transformations we are in the position to ascertain the 
characterization of the existence of coexistence states through 
the application of classical variational methods to \eqref{18}.

This decoupling technique has been previously used, and for the first time introduced in \cite{FiguMit,KlassMit}, 
to analyze several non-cooperative systems (the off-diagonal couple terms have opposite signs) such 
as the FitzHugh--Nagumo type which serve as a model for nerve conduction. In those results the authors obtained 
existence and multiplicity results as well as a spectral analysis \cite[Section 1]{FiguMit} for the linear operator
\begin{equation}
\label{opmit}
(-\Delta-\lambda) -\gamma(-\Delta-\lambda)^{-1}.
\end{equation}

It is worth mentioning that the differential operator 
involved in the system \eqref{11} and denoted by \eqref{12} is not self--adjoint 
so it is not possible to obtain its variational approach in order to ascertain such 
characterization. However, \eqref{18} does possess the associated 
variational form \eqref{19} whose critical points provide us 
with weak solutions of \eqref{18} in $H_0^1(\O)$, i.e.,
\begin{equation}
 \label{defweak11}
    \tex{\int\limits_\O \nabla u \cdot \nabla \nu -\l \int\limits_\O u\nu - \g \int\limits_\O (-\D-\l)^{-1/2} u \cdot
    (-\Delta-\l)^{-1/2} \nu  - \int\limits_\O a(x)f(x,u)u \nu =0,  }
\end{equation}
for any $\nu\in H_0^1(\O)$ (or $C_0^\infty(\O)$). 
Moreover,  
by classic elliptic regularity (Schauder's theory) those weak solutions are also classical 
solutions. 

Therefore, the existence of solutions for the  
problem \eqref{18} is ascertained through a classical variational approach and, in addition,  
under condition \eqref{17} we have that \eqref{18}
admits a positive solution if and only if \eqref{11} possesses a
coexistence state. 
\par
In order to accomplish such characterization of the existence and uniqueness 
of coexistence states we must use the spectral bound which appeared for the first time in \cite{AL4},  
and denoted by 
\begin{equation}
\label{110}
   \Sigma(\lambda) := \sup_{w \in \mathcal{P}} \inf_{\Omega_0}
  \frac{(-\Delta-\lambda)w}{(-\Delta-\lambda)^{-1}w}, \qquad
  \mathcal{P} :=\{w\in \mathcal{C}_0^2(\bar \Omega)\;:\;\; w\gg  0\}.
\end{equation} 
In this paper we shall use and introduce an equivalent variational expression defined by
\begin{equation}
\label{exva}
	\Sigma(\l):= \inf \frac{\int_{\O_0} |\nabla w|^2 -\l}
	{\int_{\O_0} \left|(-\D-\l)^{-1/2} w\right|^2},\quad \hbox{such that} \quad w\in H_0^1(\O_0) \quad \hbox{and} \quad \int_{\O} w^2 =1,
\end{equation}
more suitable to the methodology used throughout this work.

Consequently, the main result of this paper which establishes 
the existence and uniqueness of coexistence states for the problem \eqref{11} 
as well as for \eqref{18} and the limiting behaviour at the limiting values of the main continuation parameter $\g:=\a\b$ 
considered here is as follows. This result substantially improves 
\cite{AL4}.
\begin{theorem}
\label{Th 11} Suppose $\lambda<\sigma_1$. 
Then, \eqref{11} possesses a coexistence state if and only if
\begin{equation}
\label{pc}
   (\sigma_1-\lambda)^2  < \g < \Sigma(\lambda),
\end{equation}
and it is unique, if it exists, and if we denote it by
$\theta_{[\g,\O]}:=(u_\g,v_\g)$, then 
\begin{equation*}
  \lim_{\g\downarrow (\sigma_1-\lambda)^2}\theta_{[\g,\O]}=(0,0) 
  \quad \hbox{in}\;\; \mathcal{C}(\bar\O) 
  \times \mathcal{C}(\bar\O),
\end{equation*}
and
\begin{equation*}
  \lim_{\g\uparrow \Sigma(\l)}\theta_{[\g,\O]}=(\infty,\infty) \quad 
  \hbox{in}\;\;  \bar\O.
\end{equation*}
Furthermore, the map
\begin{equation*}
  \begin{array}{ccc} ((\sigma_1-\lambda)^2,\Sigma(\lambda)) & \overset{\t}\longrightarrow &
  \mc{C}(\bar \O) \times \mc{C}(\bar \O) 
  \\ \g & \mapsto & \theta(\g):=\theta_{[\g,\O]} \end{array}
\end{equation*}
is point--wise increasing and of class $\mc{C}^1$.
\end{theorem} 
The motivation for the definition of the spectral bound \eqref{110} 
comes from the fact that the nonlinearities vanish in different subdomains. 
In fact, the results shown in \cite{AL-NA} rested on the spatial assumptions considered where 
the nonlinearities vanish in the same subdomains. 
Hence, based upon a method developed in \cite{FKLM}, \cite{Lo2}, a positive supersolution was
constructed approximating the eigenfunction associated with the 
principal eigenvalue $\s[\mc{L}_0,\O_0]$ by a positive smooth extension. 
Thus, we have the characterization of coexistence states if and only if 
\begin{equation}
\label{usucon}
 \s[-\D,\O]-\sqrt{\a\b}=\s[\mf{L}_0,\O]  <  \l < \s[\mf{L}_0,\O_0]=\s[-\D,\O_0]-\sqrt{\a\b}.
\end{equation}

However, 
for the situation supposed in this paper such a construction 
is not available. Particularly, one possible justification comes from the following fact. 
Setting the spectral bound defined 
by \eqref{110} as  
\begin{equation*}
  \Sigma(\lambda) = \sup_{u\in (-\Delta-\lambda)^{-1}(\mathcal{P})}
  \inf_{\Omega_0}  \frac{(-\Delta-\lambda)^2 u}{u}.
\end{equation*}
we obtain $\Sigma(\l)$ maximizing 
\begin{equation*}
   \inf_{\Omega_0} \frac{(-\Delta-\lambda)^2 u}{u}
\end{equation*}
among all the functions $u$ of the form
\begin{equation*}
  u = (-\Delta-\lambda)^{-1}w,
\end{equation*}
for some $w \gg 0$, such that 
\begin{equation}
\label{conca}
  -\Delta u = \lambda u + w >0 \qquad \hbox{if}\;\; \lambda \geq 0,
\end{equation}
and, therefore, the functions $u_k$, $k\geq 1$, of any maximizing
sequence approximating $\Sigma(\lambda)$  must be concave.
On the other hand, it is well known that
\begin{equation*}
  (\sigma_1^0-\lambda)^2 = \sup_{u\in \mathcal{P}} \inf_{\Omega_0}
  \frac{(-\Delta-\lambda)^2 u}{u}=\sigma[(-\Delta-\lambda)^2,\Omega_0]
\end{equation*}
(see \cite[Theorem 3.1]{Lo3} and the references therein). 
Consequently, it is not possible to construct any positive smooth 
extension approximating the principal eigenfunction associated to 
$\s_1^0$ as was done in \cite{AL-NA}, \cite{FKLM}, \cite{Lo2} and 
\cite{Molina3}. So, since that approximation is not possible we can conclude that 
the next estimate should hold
\begin{equation}
\label{111}
  \Sigma(\lambda) < (\sigma_1^0-\lambda)^2.
\end{equation}
Indeed, thanks to the sharp estimations obtained for the spectral bound $\Sigma(\l)$ in \cite{AL4} 
we have that 
\begin{equation}
\label{ine}
 (\s[-\D,\O]-\l)^2  <  \Sigma(\l) \leq (\s[-\D,\O_0]-\l)^2.
\end{equation}
In addition, through the proof of those inequalities it was claimed that the profile 
where $\Sigma(\l)$ is reached does not belong to 
$\mathcal{P} :=\{w\in \mathcal{C}_0^2(\bar \Omega)\;:\;\; w\gg 0\}$. Indeed, that 
profile seems to belong to rather general classes of non-smooth functions like 
$w \in \mc{C}^2(\bar \O_0)\cup L^\infty(\O_+)$. 
Among all those functions $w$ with a fixed restriction $w|_{\O_0}$, the one
maximizing $\mf{J}(w):=\inf_{\Omega_0}
  \frac{(-\D-\lambda I)w}{(-\D-\lambda I)^{-1}w}$ is given through
\begin{equation}
\label{311}
  \tilde w := \left\{ \begin{array}{ll} w|_{\O_0} & \quad \hbox{in}
  \;\; \bar \O_0, \\ 0 & \quad \hbox{in} \;\; \O_+,\end{array}\right.
\end{equation} 
Although the exact profile where $\Sigma(\l)$ is reached still remains an open problem, it is extremely important to remark that,
in such a case, the second estimate of \eqref{ine} must be strict. We believe that the spectral properties 
for the operator \eqref{opmit} obtained in \cite[Section 1]{FiguMit} 
could help to show the path to follow in order to ascertain such a profile.

Furthermore, in this paper we obtain the limiting behaviour when the parameter $\g$ 
reaches the spectral bound $\Sigma(\l)$ after claiming that \eqref{111} is true. 
Something that is not completely proved, since it relies on \eqref{111} and, hence, on the profile \eqref{311} 
(and this is not known yet), but allows us to considerably improve 
previous results (see \cite{AL4}). 
\par
Finally, note that \eqref{18} can be regarded  as a 
\emph{non-local perturbation} of the
generalized logistic boundary value problem
\begin{equation*}
  \left\{ \begin{array}{ll}
  -\Delta u  = \lambda u - a(x) f(x,u)u & \quad \hbox{in}
  \;\; \Omega,\\ u =0 & \quad \hbox{on} \;\;
  \partial\Omega,\end{array} \right.
\end{equation*} 
by switching off to $0$ the product $\alpha\beta$ - 
the cooperative effects of \eqref{11}. According to 
\cite{FKLM}, this 
unperturbed problem possesses a positive solution if and
only if
\begin{equation*}
  \sigma_1<\lambda <\sigma_1^0.
\end{equation*}
Thus, and according to the above-mentioned discussion, the classical results of
Br\'{e}zis and Oswald \cite{BO}, T. Ouyang \cite{Ou}, and J. M.
Fraile et al. \cite{FKLM} are of a rather different nature than
those derived from this paper for the non-local problem
\eqref{18}.


\subsection{Alternative approach}
 

As an alternative to the analysis carried out in this work after performing the previously mentioned decoupling method, 
which provides us with the non-local problem \eqref{18} we might
apply standard variational arguments directly to the system \eqref{11}. To do so, one can analyze the system of the form
\begin{equation}
\label{alt1}
  \left\{ \begin{array}{ll} \begin{array}{l}
  \displaystyle{\frac{-\Delta u}{\a}=\frac{\l}{\a} u +  v - \frac{a}{\a} f(x,u) u}\\ 
  \displaystyle{\frac{-\Delta v}{\b}= u + \frac{\l}{\b} v}  \end{array}
  & \quad \hbox{in }\;\; \O,  \\ (u,v)=(0,0)
  &  \quad \hbox{on }\;\;\p \O,\\
  \end{array}\right.
\end{equation}
equivalent to system \eqref{11} after dividing the first equation of the system \eqref{11} by the parameter $\a$ and the second by the parameter $\b$. 
Thus, this system \eqref{alt1} possesses an associated functional  
\begin{equation}
\label{alt2}
	\mc{J} (u,v) =\frac{1}{2\a}\int_\O|\nabla u|^2 +  \frac{1}{2\b}\int_\O|\nabla v|^2- 
	\frac{\l}{2\a } \int_\O u^2- \frac{\l}{2\b } \int_\O v^2 - \int_\O uv +\frac{1}{\a} \int_\O 
	 F(x,u), 
\end{equation}
where 
\begin{equation*}
  F(x,u):=\int_0^{u} a(x) f(x,\xi) \xi d\xi.
\end{equation*}
Working on the functional \eqref{alt2} we can obtain similar results to those we ascertain here for the functional \eqref{19}. However, for the system \eqref{alt1} we assume 
$\l$ as the main continuation parameter. Although that is not a big issue in this case it seems to be more convenient and natural.

Moreover, note that using this alternative approach we cannot construct either a positive strict supersolution as was done in \cite{AL-NA}, 
because of the concavity shown by \eqref{conca}. However, 
using standard variational techniques one can easily deduce a similar result to Theorem\,\ref{Th 11}. In order to state such a result 
we define the Rayleigh quotient for the linear problem
\begin{equation}
\label{alt3}
      \left\{ \begin{array}{ll} \begin{array}{l}
   \frac{-\Delta u}{\a}=\frac{\s}{\a} u +  v\\ [3pt] 
  \frac{-\Delta v}{\b}= u + \frac{\s}{\b} v  \end{array}
  & \quad \hbox{in }\;\; D,  \\ (u,v)=(0,0)
  &  \quad \hbox{on }\;\;\p D,
  \end{array}\right.
\end{equation} 
as follows 
$$\s=\inf \frac{ \frac{1}{\a}\int_D |\nabla u|^2 +  \frac{1}{\b}\int_D |\nabla v|^2 - 2  \int_D uv }{ \frac{1}{\a } \int_D u^2+ \frac{1}{\b } \int_D v^2 }.$$
for a domain $D$ and representing the principal eigenvalue of the problem \eqref{alt3}. Indeed, if we assume normalized $L^2$--norms 
$$\int_D u^2=1 \quad \hbox{and}\quad \int_D v^2=1,$$ 
we have that
$$\s[\mf{L}_1,\D]:=\inf \frac{ \frac{1}{\a}\int_D |\nabla u|^2 +  \frac{1}{\b}\int_D |\nabla v|^2 - 2  \int_D uv }{ \frac{1}{\a }+ \frac{1}{\b }},$$
which corresponds to the problem
\begin{equation}
\label{alt4}
      \begin{cases}
  &  \frac{-\Delta u}{ 1+ \a/\b} - \frac{\a\b}{ \a + \b} v= \s[\mf{L}_1,D] u \\ &
  \frac{-\Delta v}{ 1+ \b/\a}-\frac{\a\b}{ \a + \b}  u=  \s[\mf{L}_1,D] v  \end{cases}
  \quad \hbox{in }\;\; D,  \qquad (u,v)=(0,0)
   \quad \hbox{on }\;\;\p D,
\end{equation} 
and the linear operator of that eigenvalue problem denoted by
 $$ \mf{L}_1 :=\left(\begin{array}{cc}  \frac{-\Delta }{ 1+ \a/\b}  & - \frac{\a\b}{ \a + \b}\\
     - \frac{\a\b}{ \a + \b} &  \frac{-\Delta }{ 1+ \b/\a} \end{array}\right).
     $$
Thus, by similar arguments as those we perform here for the functional \eqref{19} we can state the next result.
 
\begin{theorem}
\label{Th alt1}  Assume the spatial considerations $A$ and $B$ established above are satisfied.
Then, \eqref{alt1} possesses a coexistence state if and only if
\begin{equation}
\label{altpc}
   \s[\mf{L}_0,\O]  <  \l < \s[\mf{L}_1,\O_0],
\end{equation}
and it is unique, if it exists, and if we denote it by
$\theta_{[\l,\O]}:=(u_\l,v_\l)$, then 
\begin{equation*}
  \lim_{\l\downarrow \s[\mf{L}_0,\O]}\theta_{[\l,\O]}=(0,0) 
  \quad \hbox{in}\;\; \mathcal{C}(\bar\O) 
  \times \mathcal{C}(\bar\O),
\end{equation*}
and
\begin{equation*}
  \lim_{\l\uparrow \s[\mf{L}_1,\O_0]}\theta_{[\l,\O]}=(\infty,\infty) \quad 
  \hbox{in}\;\;  \bar\O.
\end{equation*}
Furthermore, the map
\begin{equation*}
  \begin{array}{ccc} (\s[\mf{L}_0,\O],\s[\mf{L}_1,\O_0]) & \overset{\t}\longrightarrow &
  \mc{C}(\bar \O) \times \mc{C}(\bar \O) 
  \\ \g & \mapsto & \theta(\l):=\theta_{[\l,\O]} \end{array}
\end{equation*}
is point--wise increasing and of class $\mc{C}^1$.
\end{theorem}

\begin{remark}
{\rm Similarly to the analysis of the functional \eqref{19} we note that analyzing the functional \eqref{alt2}, 
the upper bound for the main continuation parameter (in this case) $\l$ in the interval for the existence of positive solutions \eqref{altpc} 
is again smaller than $\s[\mf{L}_0,\O_0]$.  Indeed, problem \eqref{alt4} might be written as
$$\begin{cases}
  &  -\Delta u - \frac{\a\b+\a^2}{ \a + \b} v= \s[\mf{L}_1,D] (1+ \frac{\a}{\b})u \\ &
  -\Delta v-\frac{\a\b+\b^2}{ \a + \b}  u=  \s[\mf{L}_1,D] (1+ \frac{\b}{\a})v  \end{cases}
  \quad \hbox{in }\;\; D,  \qquad (u,v)=(0,0)
   \quad \hbox{on }\;\;\p D,
   $$
   or equivalently,
   $$\begin{cases}
  &  -\Delta u - \a v= \s[\mf{L}_1,D] (1+ \frac{\a}{\b})u \\ &
  -\Delta v-\b  u=  \s[\mf{L}_1,D] (1+ \frac{\b}{\a})v  \end{cases}
  \quad \hbox{in }\;\; D,  \qquad (u,v)=(0,0)
   \quad \hbox{on }\;\;\p D,
   $$
   for any domain $D$. 
Hence, thanks to the positivity of the cooperative terms $\a$ and $\b$ we can easily deduce that
$$\s[\mf{L}_1,\O_0]<\s[\mf{L}_0,\O_0].$$
Therefore, we again arrive at a smaller interval for the parameter than what is usually obtained in these types of problems that we denoted by \eqref{usucon}
(see \cite{AL-NA} for any further details).

Moreover, we would like to point out that when we assume $\g=\a\b$ as the main continuation parameter in Theorem\;\ref{Th 11} we obtain condition 
 \eqref{pc} for the existence of positive solutions. Equivalently, in Theorem\;\ref{Th alt1} we
arrive at condition \eqref{altpc} that, although different, both conditions seem to be equivalent since again, and strikingly, the upper bound will be smaller than the correspondent usual one \eqref{usucon}
for these types of heterogeneous problems. However, in this work we will concentrate particularly on the analysis of the functional \eqref{19}. }
\end{remark}


\subsection{Outline of the paper}


\noindent The outline of this paper is as follows. In Section 2 we collect 
some properties of the functional \eqref{19}. 
In Section 3 we ascertain the necessary conditions for the existence 
of coexistence states and in Section 4 we show the sufficient conditions 
as well as the uniqueness of coexistence states for Theorem\;\ref{Th 11}, with a general idea of the sufficient conditions for the proof of Theorem\;\ref{Th alt1} . Finally, in Section 5 
we obtain the limiting behaviour of the 
solutions when the parameter $\g$ approaches $(\s_1-\l)^2$ and $\Sigma(\l)$ 
finishing the proof of Theorem\;\ref{Th 11}.


\setcounter{equation}{0}
\section{Preliminary properties of functional $\mc{E}_\g(u)$}


\noindent For the sake of the completion, 
in this section we study some of the properties of the functional 
$\mc{E}_\g(u)$ denoted by \eqref{19}. This can be performed similarly for the functional \eqref{alt2} but here we will focus on the functional \eqref{19}.  
\par
We split the functional \eqref{19} between two in order to prove its properties. 
So, we denote it by $\mc{E}_\g(u):= \mc{E}_1(u)+\mc{E}_2(u)$ where
\begin{equation}
\label{21}	\begin{split}
	\mc{E}_1(u) & :=\frac{1}{2}\left[\int_\O |\nabla u|^2 - 
	\l\int_\O  u^2- \a\b\int_\O  
	\left|(-\Delta-\lambda)^{-1/2} u\right|^{2}\right],
	\\  
	\mc{E}_2(u) & :=\int_\O F(x,u).  
	\end{split}
\end{equation}
Once the notation is established we prove the following two 
lemmas which are well known, provide us with the regularity of  
the functionals defined by \eqref{21}. Hereafter, we are assuming that $H_0^1(\O)=
W_0^{1,2}(\O)$.

\begin{lemma}
\label{Le 21}
The functional $\mc{E}_1(u)$ is Fr\'echet differentiable 
and its Fr\'echet derivative is 
\begin{equation*}
	D_{u}\mc{E}_1(u) \nu 
	:=\int_\O \nabla u \cdot \nabla \nu - 
	\l \int_\O  u \nu - \a\b \int_\O  (-\Delta-\lambda)^{-1/2} u\cdot  
	(-\Delta-\lambda)^{-1/2} \nu,
\end{equation*}
for some $\nu \in H_0^1(\O)$.
\end{lemma}
\begin{proof}
Let $\mc{E}_1(u+\nu)$ be 
\begin{equation*}
	\mc{E}_1(u+\nu) 
	:=\frac{1}{2}[\int_\O |\nabla u+\nu|^2 
	- \l \int_\O  (u+\nu)^2 - \a\b \int_\O  
	\left|(-\Delta-\lambda)^{-1/2} (u+\nu)\right|^2].
\end{equation*}
Subsequently, operating those expressions and rearranging terms yields
\begin{align*}
	\mc{E}_1(u+\nu) &   
	:=\mc{E}_1(u)+ \mc{E}_1(\nu) + 
	\int_\O \nabla u\cdot \nabla \nu 
	- \l \int_\O  u \nu 
	\\ & - \a\b \int_\O  (-\Delta-\lambda)^{-1/2} u\cdot
	(-\Delta-\lambda)^{-1/2} \nu
\end{align*}
Since, $\mc{E}_1(\nu)$ vanishes quite radically 
\begin{equation*}
	\left|\mc{E}_1(\nu)\right| \leq K 
	\left\|\nu\right\|_{H_0^1(\O)}=
	o(\left\|\nu\right\|_{H_0^1(\O)}),	
\end{equation*}
as $\left\|\nu\right\|_{H_0^1(\O)}$ 
goes to zero and for some positive constant K, then
\begin{align*}
	|\mc{E}_1(u+\nu) -\mc{E}_1(\nu) & 
	- \int_\O \nabla u\cdot \nabla \nu + 
		\l \int_\O  u \nu \\ & - \a\b \int_\O  (-\Delta-\lambda)^{-1/2} u
	\cdot (-\Delta-\lambda)^{-1/2} \nu| 
	=	o(\left\|\nu\right\|_{H_0^1(\O)}),
\end{align*}
as $\nu \rightarrow 0$ in 
$H_0^1(\O)$. This completes the proof. 
\end{proof}
\begin{lemma}
\label{Le 22}
The functional $\mc{E}_2(u,v)$ is Fr\'echet differentiable 
and its Fr\'echet derivative is 
$$
	D_{u} \mc{E}_2(u)=\int_{\O} a(x) f(x,u) u \nu, 
$$
for some $\nu \in H_0^1(\O)$.
\end{lemma}
\begin{proof}
We know that $\mc{E}_2(u) :=\int_\O F(x,u)$ with 
$F(x,u):=\int_0^{u} a(x) f(x,\xi) \xi d\xi$. To get the expression 
of $F(x,u+\nu)$ we use Taylor's expansion in $\nu=0$. Thus,
$$
	F(x,u+\nu)=F(x,u)+F_{u}(x,u)\nu+ o(\left|\nu\right|),
$$
as $\nu\rightarrow 0$. 
Hence, for every $\e>0$ there exists 
$\d=\d(\e,x)$ such that 
$$
	\left|F(x,u+\nu)-F(x,u)-F_{u}(x,u)\nu\right| \leq \e \left|\nu\right|,
$$
for $\left|\nu\right| \leq \d$.
Therefore, since $u\in H_0^1(\O)$ 
we can conclude that
$$
	\left|\mc{E}_2(u+\nu)-\mc{E}_2(u)-D_{u}\mc{E}_2(u)\nu\right|
	= o(\left\|\nu\right\|_{H_0^1(\O)}),
$$
when $\nu$ goes to zero in $H_0^1(\O)$, which concludes the 
proof.    
\end{proof}

\vspace{0.4cm}

Consequently, we have the directional derivative (Gateaux's derivative) 
of the functional \eqref{19} as follows  
\begin{equation}
\label{24}
	\frac{d}{dt} \mc{E}_\g(u+t\nu)_{|t=0}= \left\langle \nu, D_u \mc{E}_\g(u)\right\rangle
	=	D_{u}\mc{E}_\g(u) \nu.
\end{equation}
Furthermore, due to \eqref{24} the critical points of \eqref{19} are 
weak solutions in $H_0^1(\O)$ for  
equation \eqref{18}. In other words, the Fr\'echet derivative 
obtained in Lemmas\;\ref{Le 21} and \ref{Le 22} of the 
functional \eqref{19} is going to be zero when $u$ is a weak 
solution of \eqref{18}, i.e., 
\begin{equation}
\label{25}
	D_{u}\mc{E}_\g(u) \nu=0.
\end{equation}
Hence, $u\in H_0^1(\O)$ is 
a critical point of $\mc{E}_\g(u)$ if \eqref{25} holds, 
otherwise $u$ will be called a 
regular point. The value $M\in \re$ for which  
there exists a critical point $u_0$ such that 
$\mc{E}_\g(u_0)=M$ is said to be a critical value. 
Moreover, we say that $u_0\in H_0^1(\O)$  
is a global minimum for $\mc{E}_\g$ if for every 
$u\in H_0^1(\O)$ we have that $\mc{E}_\g(u)\geq \mc{E}_\g(u_0)$.
If we consider a subset of $H_0^1(\O)$ that minimum 
is supposed to be relative. We denote the critical points of the functional $\mc{E}_\g(u)$ \eqref{19} by 
$$
 \mc{C}_\g:=\{u \in H_0^1(\O)\,:\,
    D_{u}\mc{E}_\g(u) \nu=0\}.
    $$
Thus,
$u\in \mc{C}_\g$ if and only if
$$
    \tex{\int\limits_\O |\nabla u|^2 -\l  \int\limits_\O u^2- \g \int\limits_\O \left|(-\D-\l)^{-1/2} u\right|^2 + \int\limits_\O F(x,u)=0.  }
$$

\par
The following definitions will be of extreme 
importance to get the existence of solutions for 
equation \eqref{18}. 
\begin{definition}
\label{De 23}
{\rm The map $\mc{E}:V\longrightarrow \re$, where $V$ is a Banach 
space, is weakly (sequentially) lower semicontinuous (wls) 
if for any weakly convergent sequence 
$\{u_n\}$ in $V$, $u_n \rightharpoonup u$, 
as $n\rightarrow \infty$, then
$$
	\mc{E}(u)\leq \liminf_{n\rightarrow \infty}\mc{E}(u_n)
$$
}
\end{definition}
\begin{definition}
\label{De 24}
{\rm The map $\mc{E}:V\longrightarrow \re$, where $V$ is a Banach 
space, is weakly semicontinuous (ws) 
if for any weakly convergent sequence 
$\{u_n\}$ in $V$, $u_n \rightharpoonup u$, 
as $n\rightarrow \infty$, then
$$
	\mc{E}(u)= \lim_{n\rightarrow \infty}\mc{E}(u_n)
$$
}
\end{definition}
Subsequently, after establishing the definitions for lower semicontinuity 
we easily prove that the functional $\mc{E}_\g$ denoted by \eqref{19} is (wls). 
The following result provides us with the weakly lower semicontinuity of the 
first two terms of the functional $\mc{E}_1$.
\begin{lemma}
\label{Le 25}
If X is a Hilbert space then its norm is (wls).
\end{lemma}
\begin{proof}
Since the square root function is a continuous function we find that 
\begin{equation*}
	\left\|w\right\|_{X}^2 \leq \liminf 
	\left\|w_n\right\|_{X}^2 \quad \Rightarrow \quad  
	\left\|w\right\|_{X} \leq \liminf 
	\left\|w_n\right\|_{X}, 
\end{equation*}
for any sequence $\{w_n\}$ in the space X convergent to $w\in X$.
Thus, first we assume that $w_n \rightharpoonup w$ in X and by definition    
we also have that 
\begin{equation*}
	0\leq\left\|w_n-w\right\|_{X}^2
	=\left\|w_n\right\|_{X}^2-
	2\left\langle w_n,w\right\rangle_{X}
	+\left\|w\right\|_{X}^2,
\end{equation*}
where, $\left\langle \cdot,\cdot\right\rangle_{X}$ represents 
the inner product of the Hilbert space X. Hence,
\begin{equation}
\label{26}
	2\left\langle w_n,w\right\rangle_{X}
	-\left\|w\right\|_{X}^2 
	\leq \left\|w_n\right\|_{X}^2.
\end{equation}
Moreover, owing to the convergence of the taken sequence, we can 
choose a subsequence of $\left\|w_n\right\|_{X}^2$, 
convergent to $\liminf \left\|w_n\right\|_{X}^2$. 
Therefore, passing to the limit \eqref{26} we find that
\begin{equation*}
	\left\|w\right\|_{X}^2 \leq \liminf 
	\left\|w_n\right\|_{X}^2,
\end{equation*}	
which concludes the proof. 
\end{proof}

\vspace{0.2cm}

Thus, assume $w\in H_0^1(\O)$. Consequently, applying Lemma\;\ref{Le 25} 
\begin{equation*}
	\int_\O (\left|\nabla w\right|^2-\l w^2),
\end{equation*}	
is (wls).
Since the Banach space $H_0^1(\O)$ is the 
closure of $\mc{C}_0^{\infty}(\O)$ with respect to the 
norm 
\begin{equation*}
	\left\|w_n\right\|_{H_0^1(\O)}:=\left(\int_\O 
	\left|\nabla w\right|^2+ w^2\right)^{\frac{1}{2}},
\end{equation*}
thanks to Poincar\'e's inequality, there is 
a constant $K>0$ such that 
\begin{equation*}
	K \int_\O w^2
	\leq \int_\O \left|\nabla w\right|^2,
\end{equation*}
for every $w\in H_0^1(\O)$. Hence, we can take as norm in 
$H_0^1(\O)$ the following 
\begin{equation}
\label{27}
	\left\|w\right\|_{H_0^1(\O)}:=\left(\int_\O 
	\left|\nabla w\right|^2\right)^{\frac{1}{2}},
\end{equation}
In fact, the constant K might be the principal eigenvalue $K=\s_1$,  
for $-\D$ in $\O$ under homogeneous Dirichlet boundary conditions, and  
denoted by \eqref{15} (the smallest possible one). So, after those assumptions 
and applying Lemma\;\ref{Le 25} to $H_0^1(\O)$ with the norm obtained above 
and to $L^2(\O)$ with the standard norm  
we find that the first two terms of the functional $\mc{E}_1$ are 
(wls).  
\par
Furthermore, the third term of $\mc{E}_1$ and the functional $\mc{E}_2$ are weakly semicontinuous.
\begin{lemma}
\label{Le 27}
Suppose $u\in H_0^1(\O)$. Then, $\int_\O  
	\left|(-\Delta-\lambda)^{-1/2} u\right|^{2}$ is (ws).
\end{lemma}
\begin{proof} 
As performed in the proof of Lemma\;\ref{Le 25} we take a convergent sequence 
$\{u_n\}$ in $H_0^1(\O)$ such that $u_n \rightharpoonup u$ for some $u\in H_0^1(\O)$. 
Then, 
$(-\D-\l)^{-1/2} u_n :=f_n$, with $f_n \in H_0^{1/2}(\O)$ 
equicontinuous in $H_0^{1/2}(\O)$. Then, by the compact imbedding of 
$H_0^{1/2}(\O)$ in 
$L^2(\O)$ and the Ascoli--Arzel\'a theorem 
we can extract a convergent subsequence in $L^2(\O)$ $\{u_{m_i}\}$ 
such that $u_{m_i}\rightarrow u$, as $m_i \rightarrow\infty$. Moreover, since the 
linear operator $(-\D-\l)^{-1/2}$ is compact we find that 
\begin{align*}
	u_{m_i}\rightarrow u & 
	\Rightarrow (-\Delta-\lambda)^{-1/2} u_{m_i} \rightarrow 
	(-\Delta-\lambda)^{-1/2} u \\  & \Rightarrow
	\int_\O \left|(-\Delta-\lambda)^{-1/2} u\right|^{2}
	\rightarrow \int_\O \left|(-\Delta-\lambda)^{-1/2} u\right|^{2}.
\end{align*}
This concludes the proof.
\end{proof}

\vspace{0.2cm}

\begin{lemma}
\label{Le 26}
Suppose $u\in H_0^1(\O)$. Then, $\int_\O F(x,u)$ 
is (ws).
\end{lemma}
\begin{proof}
By Fatou's Lemma and the continuity of the Nemytskii
operator $F(x,u)$ it is possible to find a convergent
subsequence $\{u_{n_i}\}$ such that 
$$\tex{F(x,u)\leq \liminf_{n_i \rightarrow \infty}
F(x,u_{n_i}),} \quad \mbox{and}$$  $$ \tex{\int\limits_\O
F(x,u) \leq \liminf_{n_{i} \rightarrow \infty}  \int\limits_\O
F(x,u_{n_i}).
 }
 $$
Therefore, $\mc{E}_2(u_{n_i}) \rightarrow \mc{E}_2(u)$ as $n_i \rightarrow\infty$, 
in $L^{\infty}(\O)$ which concludes the proof.
\end{proof}

\vspace{0.4cm}

\setcounter{equation}{0}
\section{Necessary conditions for the existence}

\noindent In this section we prove the necessary
conditions for the existence of a coexistence state. In other words, 
it provides us with the first part of the proof of Theorem\;\ref{Th 11}.

\begin{proposition}
\label{Pr 41} Suppose $a>0$, $f$ satisfies {\rm (Af)}, {\rm (Ag)}, and the
problem \eqref{11} possesses a solution $(u,v)>(0,0)$. Then, $u\gg
0$, $v\gg 0$, and
\begin{equation}
\label{41}
  0< \sigma_1-\lambda < \sqrt{\alpha\beta}.
\end{equation}
If, in addition, $a(x)$ satisfies {\rm (A) and \rm (B)} (established in the 
introduction), then
\begin{equation}
\label{42}
  0< \sigma_1-\lambda < \sqrt{\alpha\beta} < \sqrt{\Sigma(\lambda)},
\end{equation}
where $\Sigma(\lambda)$ is the spectral bound defined by \eqref{exva}.
\end{proposition}
\begin{proof}
Suppose $a>0$ and $f$ satisfies {\rm (Af)}. Let $(u,v)>(0,0)$ be a
solution of \eqref{11}. Then, according to the Maximum Principle, we
have that $u\gg 0$ and $v\gg 0$. Moreover,
\begin{equation*}
  \left(\begin{array}{cc} -\Delta +af(\cdot,u) & -\alpha\\
     -\beta & -\Delta \end{array}\right)
     \left(\begin{array}{c} u \\ v \end{array}\right)=\lambda
     \left(\begin{array}{cc} u \\ v \end{array}\right)
\end{equation*}
and, hence, by the uniqueness of the principal eigenvalue,
\begin{equation*}
  \lambda = \sigma[\mathfrak{L}(af(\cdot,u),0),\Omega].
\end{equation*} 
As $a f(\cdot,u)>0$, we find from \eqref{16} and the monotonicity of the
principal eigenvalue with respect to the potential that
\begin{equation*}
  \lambda >\sigma[\mathfrak{L}(0,0),\Omega]=
  \sigma[\mathfrak{L}_0,\Omega]=\sigma_1-\sqrt{\alpha\beta}.
\end{equation*}
Moreover, since $v=0$ on $\partial\Omega$ and
\begin{equation*}
  (-\Delta-\lambda)v = \beta u >0 \qquad \hbox{in}\;\; \Omega,
\end{equation*}
it follows, from the Maximum Principle again, that
\begin{equation*}
  0 < \sigma[-\Delta-\lambda,\Omega]=\sigma_1-\lambda,
\end{equation*}
which completes the proof of \eqref{41}.
\par
Once we know that $\lambda <\sigma_1$, it can be inferred from the
$v$-equation of \eqref{11} that
\begin{equation*}
  v = \beta (-\Delta-\lambda)^{-1} u,
\end{equation*}
and, hence, substituting it into the $u$-equation, we are driven to
\begin{equation*}
  (-\Delta-\lambda)u = \alpha\beta (-\Delta-\lambda)^{-1} u
  - af(\cdot,u)u.
\end{equation*} 
Therefore,
\begin{equation}
\label{int1}
  (-\Delta-\lambda)u = \alpha\beta (-\Delta-\lambda)^{-1}
  u\qquad \hbox{in}\;\;\Omega_0,
\end{equation}
because $a=0$ in $\Omega_0$. 
Now, let $\tilde w$ be the profile where the spectral bound is reached and denoted by \eqref{311}. Then, 
\begin{equation*}
	\Sigma(\l)= \frac{\int_{\O_0} |\nabla \tilde w|^2 -\l \int_{\O_0} \tilde{w}^2}
	{\int_{\O_0} \left|(-\D-\l)^{-1/2} \tilde w\right|^2},
\end{equation*}
since $\tilde w_{|_{\overline \O_+}} = 0$. Hence, 
\begin{equation}
\label{int2}
  \inf \frac{\int_{\O_0} |\nabla u|^2 -\l \int_{\O_0} u^2}
	{\int_{\O_0} \left|(-\D-\l)^{-1/2} u\right|^2} \leq \Sigma(\l), \quad \hbox{with} \quad u \in H_0^1(\O_0).
\end{equation}
Observe that the equality is only true when $u = \tilde w$. 
Moreover,   
multiplying \eqref{int1} by $u$ and integrating by parts in $\O_0$ gives 
\begin{equation*}
  \frac{\int_{\O_0} |\nabla u|^2 -\l \int_{\O_0} u^2}
	{\int_{\O_0} \left|(-\D-\l)^{-1/2} u\right|^2} = \alpha\beta.
\end{equation*}
Then, 
\begin{equation*}
	\inf \frac{\int_{\O_0} |\nabla u|^2 -\l \int_{\O_0} u^2 }
	{\int_{\O_0} \left|(-\D-\l)^{-1/2} u\right|^2} = \a \b \quad \hbox{with} \quad u \in H_0^1(\O_0).
\end{equation*} 
Consequently, combining it with \eqref{int2} 
\begin{equation}
\label{43}
  \g:=\alpha \beta \leq \Sigma(\lambda).
\end{equation}
To conclude the proof we show the next lemma that actually 
sharpens \eqref{43} up to condition \eqref{42} combining it with 
\eqref{41}.
\begin{lemma}
\label{Le 42} Suppose $a>0$ in $\Omega$, $f \in
\mathcal{C}^{\mu,1+\mu}(\bar\Omega\times [0,\infty))$ satisfies
{\rm (Af)}, {\rm (Ag)} and \eqref{11} possesses a coexistence state. Then, there
exists $\varepsilon>0$ such that the perturbed problem
\begin{equation}
\label{45}
  \left\{ \begin{array}{ll} \begin{array}{l}
  -\Delta u=\lambda u + (\alpha+t) v - a f(\cdot,u)u\\
  -\Delta v=\beta u + \lambda v  \end{array}
  & \quad \hbox{in }\;\; \Omega,  \\ (u,v)=(0,0)
  &  \quad \hbox{on }\;\;\partial \Omega.\\
  \end{array}\right.
\end{equation}
has a coexistence state for every $t \in [0,\varepsilon)$.
\end{lemma}

\vspace{0.4cm}

\noindent \begin{proofof} It consists of a simple application of the
Implicit Function Theorem  based on the fact that any coexistence
state of \eqref{11} is \emph{non-degenerate}. Let $(u_0,v_0)$ be a
coexistence state of \eqref{11} and consider the operator
\begin{equation*}
  \mf{F}: E:= \mc{C}_0^{2+\mu}(\bar \O)\times \mc{C}_0^{2+\mu}(\bar
  \O)\times \re   \longrightarrow F:= \mc{C}^{\mu}(\bar \O)\times
  \mc{C}^{\mu}(\bar \O)
\end{equation*}
defined by
\begin{equation*}
  \mf{F}(u,v,t):= \left( \begin{array}{l}
  -\D u -\l u -(\a+t)v + a f(\cdot,u)u \\
  -\D v -\l v - \b u \end{array} \right),\qquad (u,v,t)\in E.
\end{equation*}
$\mf{F}$ is of class $\mc{C}^1$ and, by definition,
\begin{equation*}
  \mf{F}(u_0,v_0,0)=0.
\end{equation*}
Moreover, the differential operator
\begin{equation*}
  D_0 \mf{F}:= D_{(u,v)} \mf{F}(u_0,v_0,0)
  \in \mc{L}\left( \mc{C}_0^{2+\mu}(\bar \O)\times \mc{C}_0^{2+\mu}(\bar
  \O); \mc{C}^{\mu}(\bar \O)\times \mc{C}^{\mu}(\bar \O)\right)
\end{equation*}
is given by
\begin{align*}
  D_0 \mf{F}\binom{u}{v} & = \left( \begin{array}{cc}
  -\D  -\l  + a \p_u f(\cdot,u_0)u_0+ a f(\cdot,u_0) & -\a \\
  -\b & -\D  -\l   \end{array} \right)\left( \begin{array}{c}
  u \\  v \end{array} \right) \\ & =
  \mf{L}\left(-\l+ a \p_u f(\cdot,u_0)u_0+ a f(\cdot,u_0),-\l\right)
  \left( \begin{array}{c}  u \\  v \end{array} \right),
\end{align*}
where $\mf{L}(\cdot,\cdot)$ stands for the linear \emph{cooperative
operator} defined in \eqref{14}. According to assumptions (A) and (B), 
{\rm (Af)}, and {\rm (Ag)}, we find from the monotonicity of
the principal eigenvalue with respect to the potential that
\begin{align*}
  \s[\mf{L}\left(-\l \right. \!\! & \left. + a \p_u f(\cdot,u_0)u_0+ a
  f(\cdot,u_0),-\l\right),\O] \\ & = \s[\mf{L}\left(a \p_u f(\cdot,u_0)
  u_0+ a  f(\cdot,u_0),0\right),\O]-\l \\ & > \s[\mf{L}\left(a
  f(\cdot,u_0),0\right),\O]-\l =0.
\end{align*}
Therefore,  
the linearized operator
$D_0\mf{F}$ is an isomorphism with strong positive inverse, and,
consequently, thanks to the Implicit Function Theorem, there exist
$\e>0$ and two maps of class $\mc{C}^1$
\begin{equation*}
  U, V : (-\e,\e) \mapsto \mc{C}_0^{2+\mu}(\bar \O)\times
  \mc{C}_0^{2+\mu}(\bar \O)
\end{equation*}
such that
\begin{equation*}
  U(0)=u_0, \qquad V(0)=v_0,
\end{equation*}
and
\begin{equation*}
  \mf{F}(U(t),V(t),t)=0\qquad \hbox{for every}\;\; t \in (-\e,\e).
\end{equation*}
As $u_0$ and $v_0$ lie in the interior of the cone of positive
functions of the ordered Banach space $\mc{C}_0^1(\bar \O)$, it
becomes apparent that $(U(t),V(t))$ is a coexistence state of
\eqref{45} for sufficiently small $t>0$. This completes the proof.
\end{proofof}

\vspace{0.4cm}

Consequently, owing to \eqref{43} we find that $\a\b = \Sigma(\l)$ if 
\eqref{42} fails. Moreover, by  the analysis already
done in this proof,
\begin{equation*}
  (\a+t)\b \leq \Sigma(\l) \qquad \forall\;\; t \in [0,\e).
\end{equation*}
This contradiction shows that actually $\g:=\a\b<\Sigma(\l)$ and concludes the proof.
\end{proof}


\setcounter{equation}{0}
\section{Sufficient conditions for the existence}


\noindent As discussed in the introduction of this paper 
equation \eqref{18} admits positive solutions if and 
only if the system \eqref{11} possesses coexistence states. Since 
the operator \eqref{12} is not self-adjoint we ascertain the 
characterization of the positive solutions for equation \eqref{18}, 
for which a variational setting is guaranteed. That result provides 
us with the final characterization of coexistence states of \eqref{11}.
\par
In this context, the differential equation \eqref{18} can be viewed as 
the Euler-Lagrange equation represented by the functional 
denoted by \eqref{19} when \eqref{17} is satisfied.
\par
The next result is pivotal in ascertaining the characterization 
of the existence coexistence states of \eqref{11} and as discussed above, 
the existence of positive solutions of \eqref{18}. It provides 
us with the coercivity of the functional \eqref{19}. 
\begin{lemma}
\label{Le 51}
Suppose \eqref{pc} is satisfied then 
the functional $\mc{E}_\g(u)$ defined by \eqref{19} is coercive.
\end{lemma}
\begin{proof}
We argue by contradiction. Then, suppose \eqref{19} is 
not coercive, i.e.,
\begin{equation}
\label{prin}
	\mc{E}_\g(u_n) \leq C,
\end{equation}
such that there exists a sequence $\{(u_n)\}$ for 
which 
\begin{equation}
\label{51}
	\left\|u_n\right\|_{H_0^1(\O)}\rightarrow \infty
\end{equation}
as $n \rightarrow\infty$, holds. Note that we also have that 
$\left\|u_n\right\|_{L^2(\O)}\rightarrow \infty$, as $n \rightarrow\infty$, from 
\eqref{51} and the structure of the non-local 
compact operator. Indeed, to prove it we suppose that $\left\|u_n\right\|_{L^2(\O)}$ is bounded 
for any $n\geq 1$. Then, since $(-\D-\l)^{-1/2}$ is a compact operator from $L^2(\O)$ to itself we find that 
$$
	\left\|(-\D-\l)^{-1/2} u_n\right\|_{L^2(\O)}\leq C \left\|u_n\right\|_{L^2(\O)},
$$
for a positive constant $C>0$. Hence, since $\mc{E}_\g(u_n)$ is bounded \eqref{prin} we find that 
$$
	\int_\O |\nabla u_n|^2 \leq   C+ 
	\a\b \int_\O \left|(-\D-\l)^{-1/2} u_n\right|^2 + \l \int_\O \left|u_n\right|^2\leq C + K \int_\O \left|u_n\right|^2,
$$
for some positive constant $K>0$. This obviously, owing to \eqref{51}, 
contradicts our assumption about the boundedness of $\left\|u_n\right\|_{L^2(\O)}$ for any $n\geq 1$. 
Therefore, 
$$\left\|u_n\right\|_{L^2(\O)}\rightarrow \infty\quad \hbox{as} \quad n \rightarrow\infty.$$
\par
Furthermore, just remember that according to \eqref{prin} $\mc{E}_\g(u_n)\leq C$, for any $n\geq 1$,  
and some positive constant $C>0$. Thus, we can ensure that  
$$
	\limsup_{n \rightarrow\infty} \frac{\mc{E}_\g(u_n)}{\left\|(u_n)\right\|_{L^2(\O)}^2} 
	\leq 0
$$
and, hence, 
\begin{equation}
\label{57}
	\limsup_{n \rightarrow\infty}  \frac{1}{2}\int_\O |\nabla w_n|^2 
	- \frac{\l}{2} 
	+ 	\int_\O \frac{F(x,u_n)}{\left\|u_n\right\|_{L^2(\O)}^2} - \frac{\a\b}{2} \int_\O \left|(-\D-\l)^{-1/2} w_n\right|^2 
	\leq 0,
\end{equation}
with
\begin{equation}
\label{53}
	w_n:= \frac{u_n}{\left\|u_n\right\|_{L^2(\O)}}. 
\end{equation} 
Then, $\left\|w_n\right\|_{L^2(\O)}=1$.

On the other hand, owing to \eqref{17} and the fact that 
$\int_\O \left|(-\D-\l)^{-1/2} w_n\right|^2\leq K$ is bounded, since 
the non-local operator is compact and $\{w_n\}$ is bounded sequence,  
we have that 
\begin{equation}
\label{54}
	\int_\O |\nabla w_n|^2\leq K\quad \hbox{and}\quad 
	\int_\O \frac{F(x,u_n)}{\left\|u_n\right\|_{L^2(\O)}^2} \leq K,
\end{equation}
for some positive constant $K>0$.
Then, $\{w_n\}$ is a bounded sequence in $H_0^1(\O)$ hence it converges weakly in $H_0^1(\O)$. Moreover, as the imbedding 
\begin{equation*}
  H_0^1(\O)\hookrightarrow L^2(\O)
\end{equation*}
is compact, for every $n\geq 1$, due to the normalization 
of $w_n$ in $L^2(\O)$ and \eqref{54}  
there exists a subsequence of $\{w_n\}_{n\geq 1}$,
again labelled by $n$, and $w_0\in L^2(\O)$ such that
\begin{equation}
\label{55}
 \lim_{n\to\infty} \|w_{n}-w_0\|_{L^2(\O)}=0.
\end{equation} 
In addition, we will prove that $\{w_{n}\}_{n\geq 1}$ is  actually a
Cauchy sequence in $H_0^1(\O)$ using an argument shown in \cite{AL-JDE}. This implies that $w_0\in H_0^1(\O)$
and
\begin{equation}
\label{b45}
    \lim_{n\to \infty}\|w_{n}-w_0\|_{H_0^1(\O)}=0.
\end{equation}
Indeed, for every $n<m$, we have that $w_{n} < w_{m}$ and 
\begin{align*}
      \int_{\O} \left|\nabla (w_{n} - w_{m})\right|^2 & =
    \int_{\O}\left|\nabla w_{n}\right|^2+
    \int_{\O}\left|\nabla w_{m} \right|^2- 2
    \int_{\O}\left\langle \nabla w_{n} ,\nabla w_{m}\right\rangle \\
     & =\l (\int_{\O} w_{n}^2 + \int_{\O} w_{m}^2- 2 \int_{\O} w_{n}w_{m}) 
     + \a\b \int_\O \left|(-\D-\l)^{-1/2} w_n\right|^2 \\ & + 
     \a\b\big( \int_\O \left|(-\D-\l)^{-1/2} w_m \right|^2
     - 2 \int_\O w_{m} (-\D-\l)^{-1} w_n \big) 
    \\ & - \int_\O a(x) f(x,u_n) w_n^2-\int_\O a(x) f(x,u_m) w_m^2\\ & + 
    2  \int_\O a(x) f(x,u_n)  w_n w_m .
\end{align*}
Thus, rearranging terms and due to the monotonicity of the 
function $f$, supposed by the assumption (Af), and the final discussion 
of section 2 we are driven to the inequality
\begin{align*}
	\int_{\O} \left|\nabla (w_{n} - w_{m})\right|^2 & =
	\l \int_{\O}(w_{n}-w_{m})^2 
	-\int_\O a(x) f(x,u_n) (w_{n}-w_{m})^2
	\\ & + C_1 \int_{\O}(w_{n}-w_{m})^2 \leq \l \int_{\O}(w_{n}-w_{m})^2 + C_1 \int_{\O}(w_{n}-w_{m})^2,
\end{align*}
for a positive constant $C_1>0$ whose specific value is not important.
Consequently, according to H\"{o}lder's inequality and the fact that 
the sequence $\{w_n\}$ is already a Cauchy sequence in $L^2(\O)$ 
it becomes apparent that $\{w_n\}_{n\geq 1}$ is a Cauchy sequence in
$H_0^1(\O)$ and, therefore, $w_0 \in H_0^1(\O)$ and \eqref{b45}
holds. Note that,
\begin{equation}
\label{b47}
  w_0 \geq 0 \qquad \hbox{and}\qquad \int_\O w_0^2 =1.
\end{equation}
Moreover, from the fact that $\left\|u_n\right\|_{L^2(\O)}\rightarrow \infty$, as $n \rightarrow\infty$, and  
thanks to \eqref{54} and Fatou's Lemma 
we find that
\begin{equation}
\label{56}
  w_0 =0 \qquad \hbox{in}\quad
  \O_+=\{x\in\O:a(x)>0\}.
\end{equation}
Passing to the limit in \eqref{57} as $n \rightarrow\infty$ gives  
\begin{equation}
\label{p33}
	\frac{\int_{\O_0} |\nabla w_0|^2 -\l}{\int_{\O_0} \left|(-\D-\l)^{-1/2} w_0\right|^2} 
	\leq \a\b,
\end{equation}
since $w_0=0$ in $\O_+ = \O\setminus \O_0$ by \eqref{56}. Furthermore, we are supposing that 
$\a\b <\Sigma(\l)$
which in its variational expression means that 
$$\a\b <\Sigma(\l):= \inf \frac{\int_{\O_0} |\nabla w|^2 -\l}{\int_{\O_0} \left|(-\D-\l)^{-1/2} w\right|^2}
\quad \hbox{such that} \quad w\in H_0^1(\O_0) \quad \hbox{and} \quad \int_{\O_0} w^2 =1.$$
Thus, for any sequence that fulfills \eqref{51} we find that the functional 
$\mc{E}_\g$ must be bounded if \eqref{p33} is satisfied so it must be also 
true for the supremum  
which clearly contradicts \eqref{pc}. 
Therefore, the functional $\mc{E}_\g$ is 
coercive for that range of $\g=\a\b$ for which condition \eqref{pc} is 
satisfied. That completes the proof.
\end{proof} 

\begin{remark}
{\rm The proof of coercivity for the functional \eqref{alt2} will follow a similar 
argument with a couple of differences. First, to prove the convergence of a sequence 
$\{w_{n,1},w_{n,2}\}_{n\geq 1}$ is actually a
Cauchy sequence in $H_0^1(\O)\times H_0^1(\O)$ such that 
$$w_{n,1}:= \frac{u_n}{\left\| (u_n,v_n)\right\|_{L^2(\O)\times L^2(\O)}}\quad w_{n,2}:= \frac{v_n}{\left\| (u_n,v_n)\right\|_{L^2(\O)\times L^2(\O)}}$$
we can use an argument shown in \cite{AL-JDE} in which such a convergence is obtained for a class of cooperative systems  such as \eqref{alt1}, having
that 
\begin{equation}
\label{limalt}
    \lim_{n\to \infty}\|(w_{n,1},w_{n,2})-(w_{0,1},w_{0,2})\|_{H_0^1(\O)\times H_0^1(\O)}=0.
\end{equation}
with $(w_{0,1},w_{0,2})\in H_0^1(\O)\times H_0^1(\O)$, $w_{0,1},w_{0,2}\geq 0$,
\begin{equation}
\label{nullcon}
 w_{0,1} =w_{0,2}=0 \qquad \hbox{in}\quad
  \O_+=\{x\in\O:a(x)>0\}.
  \end{equation}
Indeed, assuming that
$$\mc{J}_\l(u_n,v_n)\leq C,\quad \hbox{for any}\quad n\geq 1,$$  
and some positive constant $C>0$, if 
$$\left\| (u_n,v_n)\right\|_{L^2(\O)\times L^2(\O)} \rightarrow \infty\quad \hbox{as} \quad n \rightarrow\infty,$$
 we can ensure that  
$$
	\limsup_{n \rightarrow\infty} \frac{\mc{J}_\l(u_n,v_n)}{\left\| (u_n,v_n)\right\|_{L^2(\O)\times L^2(\O)}^2} 
	\leq 0
$$
and, hence, 
\begin{equation}
\label{boundfun}
\begin{split}
	\limsup_{n \rightarrow\infty}  \frac{1}{2\a}\int_\O |\nabla w_{n,1}|^2 & + \frac{1}{2\b}\int_\O |\nabla w_{n,2}|^2  
	- \frac{\l}{2\a}  \int_\O  w_{n,1}^2 -  \frac{\l}{2\b}  \int_\O  w_{n,2}^2 \\ & - \int_\O w_{n,1} w_{n,2} 
	+ \frac{1}{\a}	\int_\O \frac{F(x,u_n)}{\left\| (u_n,v_n)\right\|_{L^2(\O)\times L^2(\O)}} 
	\leq 0,
\end{split}
\end{equation}
Then, similarly as done above for the functional \eqref{19}, and thanks to the convergence \eqref{limalt} we can easily see that 
$$w_{0,1}=0  \qquad \hbox{in}\quad
  \O_+=\{x\in\O:a(x)>0\},$$
  since from \eqref{boundfun} and the bounded norms in $L^2(\O)\times L^2(\O)$, for the sequence $(w_{n,1},w_{n,2})$, we find that
  $$\int_\O \frac{F(x,u_n)}{\left\| (u_n,v_n)\right\|_{L^2(\O)\times L^2(\O)}}\leq K,$$
  for some positive constant $K$. In fact, we have again a bounded sequence in $H_0^1(\O)\times H_0^1(\O)$. 
  Moreover, due to \cite[Lemma 3.6]{AL-NA} which is a consequence of the cooperative character of the system \eqref{11}, we actually have \eqref{nullcon}.
  In other words, either both components are strictly positive or they vanish in the same regions.
Thus, arguing by contradiction as above and assuming a normalization of the form
$$\int_\O w_{0,1}^2 =1\quad \hbox{and}\quad \int_\O w_{0,2}^2 =1,$$
we arrive at the expression
$$ \frac{ \frac{1}{\a}\int_{\O_0} |\nabla w_{0,1}|^2 +  \frac{1}{\b}\int_{\O_0} |\nabla w_{0,2} |^2 - 2  \int_{\O_0} w_{0,1} w_{0,2} }{ \frac{1}{\a }+ \frac{1}{\b }} \leq \l,$$
which contradicts condition \eqref{altpc} in Theorem\;\ref{Th alt1}, since
$$\s[\mf{L}_1,\O_0] = \inf \frac{ \frac{1}{\a}\int_{\O_0} |\nabla w_{1}|^2 +  \frac{1}{\b}\int_{\O_0} |\nabla w_{2} |^2 - 2  \int_{\O_0} w_{1} w_{2} }{ \frac{1}{\a }+ \frac{1}{\b }}.$$
}

\end{remark}

Lemma\;\ref{Le 51} ensures us that the minimizing sequence will be bounded 
under those restrictions in $H_0^1(\O)$. Now, we prove that indeed, 
the minimizer is attained. Hence, the existence of a 
weak solution is achieved when the cooperative effects $\g:=\a\b$ 
fulfill \eqref{pc}. By elliptic 
regularity we can obtain the existence of a classical solution as well. 

\begin{proposition}
\label{Pr 52}
Suppose the functional $\mc{E}_\g(u)$ defined by \eqref{19} is (wls), coercive and  
condition \eqref{pc} is satisfied. 
Then, there exists a positive minimizer $u_0>0$  
which is indeed attained.
\end{proposition}
\begin{proof}
We argue by contradiction. Suppose the functional is not bounded below, i.e., 
$\mc{E}_\g(u_n)< -n$, for any convergent sequence $\{u_n\}$ in $X$. 
Since, that sequence converges weakly in $H_0^1(\O)$ we have that
\begin{equation}
\label{59}
	\left\|u_n\right\|_{H_0^1(\O)}\leq K.
\end{equation}
for a positive constant $K>0$. Hence, there exists 
a convergent subsequence such that $u_{n_m} \rightharpoonup u$,
as $m \rightarrow\infty$, for some $u\in H_0^1(\O)$. 
However, due to the fact that the functional $\mc{E}_\g$ is (wls) we obtain that 
$$
	\mc{E}_\g(u_0) \leq  \liminf_{m\rightarrow \infty}\mc{E}_\g(u_{n_m})<-\infty,
$$
which contradicts \eqref{59}. 
Consequently, the limit exists and thanks to the coercivity of the functional 
when $0<(\sigma_1-\lambda)^2  < \alpha\beta=\g < \Sigma(\lambda)$ is finite,
$$
	\inf_{u\in H_0^1(\O)} \mc{E}_\g(u) = d < \infty.
$$
Thus, taking a minimizing sequence $\{u_n\}$, bounded 
because of the coercivity, yields
$$
	\lim_{n \longrightarrow \infty} \mc{E}_\g(u_n)= d.
$$
Then, a convergent subsequence might be chosen, such that 
$u_{n_m}\longrightarrow u_0$, 
as $m \rightarrow\infty$. 
Hence, 
$$
	\mc{E}_\g(u_0) \leq  \liminf_{m\rightarrow \infty}\mc{E}_\g(u_{n_m})=
	\lim_{n\rightarrow \infty}\mc{E}_\g(u_n)=d=\inf_{u\in H_0^1(\O)} \mc{E}_\g(u)
$$
so that $\mc{E}_\g(u_0)=d$. Moreover, since $\mc{E}_\g(0)=0$, taking a constant  
$M$ sufficiently close to $0$ and thanks to $(Af)$ we find that 
$$
	\mc{E}_\g(M)=-\frac{\l}{2} M^2 \left|\O\right|+  \int_\O F(x,M)
	-\frac{1}{2} \a\b M^2 \left|\O\right|< 0=\mc{E}_\g(0).
$$
Therefore, the minimizer $u_0 \in H_0^1 (\O)$ 
is not identically zero. Indeed a positive critical point of 
$\mc{E}_\g$ exists. This completes the proof.
\end{proof}

Furthermore, to prove the uniqueness of the
positive solutions of \eqref{18} and, hence, 
the coexistence states of \eqref{11} we go back to 
\cite[Lemma\;3.7]{AL-NA}. Thus, we again proceed by contradiction. 
Suppose \eqref{18} has two positive solutions such that 
$u_1\neq u_2$. Then, $w:=u_2-u_1>0$ and  
\begin{equation}
\label{510}
  \left\{ \begin{array}{ll}
  \left(-\Delta -\lambda - \alpha\beta (-\Delta-\lambda)^{-1}
  +V \right)w =0& \quad \hbox{in} \;\; \Omega,
  \\ w =0 & \quad \hbox{on} \;\; \partial\Omega,\end{array} \right.
\end{equation}
where  $V$ is given through
$$
      V := a \int_0^1  \frac{\p f}{\p u}(\cdot,tu_2 +(1-t)u_1) (tu_2+(1-t)
    u_1)\,dt +a \int_0^1 f(\cdot,tu_2+(1-t)u_1)\,dt.
$$
By the Maximum Principle,  $u_1\gg 0$  and  $u_2\gg 0$. 
Thus, (A) and (B) imply
$$
  V  >a \int_0^1 f(\cdot,tu_2+(1-t)u_1)\,dt\geq a
  f(\cdot,u_1),
$$
since  $u_2>u_1$. Therefore, by the monotonicity of the principal eigenvalue with
respect to the potential, we find that
$$
  \s\left[-\D-\l -\a\b  (-\Delta-\lambda)^{-1} + V ,\O\right]
  >\s\left[-\D-\l -\a\b  (-\Delta-\lambda)^{-1} + a f(\cdot,u_1),\O\right]=0.
$$
On the other hand, by \eqref{510}, $w\gg 0$ provides
us with an eigenfunction of 
$$-\D-\l -\a\b  (-\Delta-\lambda)^{-1} + V,$$
associated with the eigenvalue 0 and, consequently,
$$
    \s\left[-\D-\l -\a\b  (-\Delta-\lambda)^{-1} + V ,\O\right]=0;
$$
leading to a contradiction which ends the proof of the uniqueness.
\par

\setcounter{equation}{0}
\section{Limiting behaviour at the values $(\s_1-\l)^2$ and $\Sigma(\l)$}

\noindent In this section we analyze the limiting behaviour of the 
positive solution of the problem \eqref{18} when the 
parameter $\g:=\a\b$ approaches 
the limiting values for which the existence of positive solutions is 
held. The parameter $\g$ represents the cooperative effects between 
the components of the cooperative system \eqref{11}. So, under condition 
\eqref{17} we also ascertain the limiting behaviour for the (unique) 
coexistence state.
\par
Fixed $\gamma$ as the main continuation parameter, $u(\gamma)$ is regarded as the unique positive solution of
\eqref{18}. Then, it provides us with a zero of 
the operator
$$
  \mf{F}: E:= \mc{C}_0^{2+\mu}(\bar \O)\times \re   
  \longrightarrow   \mc{C}^{\mu}(\bar \O)
$$
defined by 
\begin{equation}
\label{61}
  \mf{F}(u,\g):=  (-\D-\l) u -\g (-\D-\l)^{-1} u + a f(\cdot,u)u,\qquad (u,\g)\in E.
\end{equation}
Moreover, as soon as $a>0$ in $\O$ and \eqref{18} admits a positive solution 
by the Implicit Function Theorem (IFT) we find that $\g \rightarrow u(\g)$ is 
a mapping of class $\mc{C}^1$ and increasing. Actually, applying the 
IFT we can differentiate the identity 
$$
  \mf{F}(u(\g),\g)=0
$$
with respect to $\g$ obtaining  that
$$
  D_{u}\mf{F}(u(\g),\g)D_\g u(\g)+
  D_\g \mf{F}(u(\g),\g)=0, 
$$ 
where $D_{u}\mf{F}(u(\g),\g)=(-\D-\l) -\g (-\D-\l)^{-1} + a f(\cdot,u)+a\p_u f(\cdot,u)$ 
and, hence, by assumptions (A), (B), (Af) and (Ag)  
we find that 
\begin{align*}
  \s[-\D-\l -\g (-\D-\l)^{-1} & + a f(\cdot,u)+a\p_u f(\cdot,u),\O] 
  \\ & \geq \s[-\D-\l -\g (-\D-\l)^{-1} + a f(\cdot,u),\O]=0,
\end{align*} 
just applying the monotonicity 
of the principal eigenvalue with respect to the potential.
Then, the operator $D_{u}\mf{F}(u(\g),\g)$ is an isomorphism and, hence, invertible 
in the way that its inverse $(D_u \mf{F})^{-1}$ is strongly positive. 
Moreover, differentiating with respect to $\g$ the operator \eqref{61}
yields
$$
  D_\g \mf{F}(u(\g),\g)=-(-\D-\l)^{-1}u(\g).
$$
Hence, since $(-\D-\l)^{-1}$ is also a positive operator we obtain 
\begin{equation}
\label{62}
  D_\g u(\g)= \left( D_\g \mf{F}\right)^{-1} (-\D-\l)^{-1}u(\g) \gg 0.
\end{equation}
In particular, regarding $\g$ as the main continuation parameter, 
the structure of the positive solutions of \eqref{18} consists 
of an increasing curve of class $\mc{C}^1$ 
$$
  \g \mapsto u(\g), \qquad \hbox{with} \quad 0<(\s_1-\l)^2< \g < \Sigma(\l).
$$
The next result provides us with the limiting behaviour 
of the positive solution when the parameter $\g$ approximates 
the external values of the interval of existence 
$$\mc{I}:=((\s_1-\l)^2,\Sigma(\l)).$$

\begin{proposition}
\label{Pr 61}
Suppose $a(x)$ satisfies the assumptions (A) and (B), and $f$ satisfies (Af) and (Ag). Then,
\begin{equation}
\label{63}
  \lim_{\g\downarrow (\s_1-\l)^2} u(\g)=0,
\end{equation}
and
\begin{equation}
\label{64}
  \lim_{\g\uparrow \Sigma(\l)} \|u(\g)\|_{_{\mc{C}(\bar \O)}} =\infty.
\end{equation}
Indeed, supposing that there exists $w\in \mc{C}_0^2(\bar\O_0)$ with 
$w|_{\p\O}=0$ and $w|_{\O_0}>0$ such that $\hat{w}$ is 
the function defined by \eqref{311} where $\Sigma(\l)$ is reached then,
\begin{equation}
\label{65}
  \lim_{\g\uparrow \Sigma(\l)} u(\g)=\infty, \quad \hbox{uniformly in compact subsets 
  of} \quad \bar\O.
\end{equation}
\end{proposition}
\begin{proof}
The zeros of the functional 
$\mf{G} : \re \times \mc{C}_0(\bar \O) \rightarrow \mc{C}_0(\bar \O)$ 
denoted by
$$
	\mf{G}(\g,u):=u - (-\D-\l)^{-1}
	\left[ \g (-\D-\l)^{-1} u -af(\cdot,u)u\right],
$$
are fixed points of a compact operator $(-\D-\l)^{-1}$
(cf. \cite[Chapter 7]{LGB}),  
where $(-\D-\l)^{-1}$ stands for the inverse of $(-\D-\l)$ in $\O$ 
under homogeneous Dirichlet boundary conditions. The functional  
$\mf{G}$ is of class $\mc{C}^1$ and by elliptic regularity 
a compact perturbation of the identity for every $\g \in\re$. Moreover, 
$\mf{G}(\g,0)=0$ for all $\g\in\re$ and, also, by (Af)
$$
	D_u \mf{G}(\g,0)u=u-\g (-\D-\l)^{-2} u, \quad \g\in\re,\quad 
	u\in \mc{C}_0(\bar \O).
$$
Thus, the linear operator $D_u \mf{G}(\g,0)$ is Fredholm of index zero 
and is analytic in $\g$, for it is a compact perturbation 
of the identity of linear type with respect to $\g$. Moreover, 
its spectrum consists of the eigenvalues of $(-\D-\l)^{2}$. In 
particular,
$$
	N\left[D_u \mf{G}((\s_1-\l)^2,0)\right]=\mathrm{span}[\phi_1],
$$
where $\phi_1 \gg 0$ is any principal eigenfunction of 
$\s_1$. Then, the following transversality condition 
of Crandall--Rabinowitz \cite{CR,CRs} holds,
\begin{equation}
\label{66}
	D_\g D_u \mf{G}((\s_1-\l)^2,0) \phi_1 \notin R[D_u \mf{G}((\s_1-\l)^2,0)].
\end{equation}
To prove \eqref{66} we argue by contradiction assuming that 
$$
	D_\g D_u \mf{G}((\s_1-\l)^2,0) \phi_1
	= -(-\D-\l)^{-2}\phi_1 \in R[D_u \mf{G}((\s_1-\l)^2,0)].
$$
Then, there exists $u\in \mc{C}_0(\bar\O)$ such that 
$$
	u-(\s_1-\l)^2 (-\D-\l)^{-2} u = -(-\D-\l)^{-2}\phi_1,
$$
and, hence,  
$$
	(-\D-\l)^{2}u-(\s_1-\l)^2 u = -\phi_1.
$$
Now, multiplying by $\phi_1$, integrating in $\O$ and applying 
the formula of integrating by parts gives $\int_\O \phi_1^2=0$, 
which is impossible. Therefore, \eqref{66} is actually 
true. Consequently, according to the main theorem of Crandall--Rabinowitz 
\cite{CR} $(\g,u)=((\s_1-\l)^2,0)$ is a bifurcation 
point from the branch of trivial solutions $(\g,u)=(\g,0)$ from 
which a smooth curve of positive solutions emanates. Indeed, that 
continuum of positive solutions emanating from the bifurcation 
point as was seen above is of class $\mc{C}^1$ and increasing pointwise 
with respect to the parameter $\g$. Note that after the 
characterization result obtained in the previous sections \eqref{18} 
cannot admit a positive solution if $\g \leq (\s_1-\l)^2$.
\par
Subsequently, to prove \eqref{64} we apply a compactness argument 
shown in \cite[Chapter 7]{LGB}. Then, by the
monotonicity of $\g \mapsto u(\g)$ and arguing by contradiction, 
there exists a constant $C>0$ such that
\begin{equation}
\label{67}
  u(\g) \leq C \qquad \hbox{in}\;\; \O, \qquad \forall \;\; 
  \g \in ((\s_1-\l)^2,\Sigma(\l)).
\end{equation} 
Hence, let $\{\g_n\}_{n\geq 1}$ be an increasing sequence 
such that $0<\g_n<\g_m$ if $n<m$ and
$$
  \lim_{n\to \infty}\g_n=\Sigma(\l).
$$
Then, take 
a convergent sequence $\{u_n\}_{n\geq 1}$, such that 
$u_n \rightarrow u_*\leq C $ as $n \rightarrow \infty$. 
So, multiplying \eqref{18} by $u_n$, integrating in $\O$ and 
applying the formula of integrating by parts gives 
$$
	\int_\O |\nabla u_n|^2 = \l \int_\O u_n^2 +\g_n \int_\O |(-\D-\l)^{-1/2}u_n|^2 
	-\int_\O a(x)f(x,u_n) u_n^2,
$$
and thanks to \eqref{67}, we find that $\int_\O |\nabla u_n|^2\leq K$, for some positive 
constant $K>0$. Then, by Agmon--Douglis--Nirenberg $L^p$ estimations  
$\left\|\nabla u_n\right\|_{L^\infty}\leq K$, for any $n \geq 1$. 
Hence, taking $x,y\in\O$ sufficiently close such that 
$\left\|x-y\right\|\leq \frac{\e}{K_1}$ 
for some $\e>0$ sufficiently small and a positive constant $K_1>0$ 
we have that 
\begin{align*}
	|u_n(x)  -u_n(y) | & =| \int_0^1 \frac{d}{dt} u_n(tx+(t-1)y)dt|
	\leq \int_0^1 |\left\langle \nabla u_n(tx+(t-1)y),x-y \right\rangle| dt
	\\ & \leq \int_0^1 \left\|\nabla u_n(tx+(t-1)y)\right\|_{L^\infty} 
	\left\| x-y \right\|_{L^\infty}dt \leq K_1 \left\| x-y \right\|_{L^\infty}
	\leq \e. 
\end{align*}
Consequently, $\{u_n\}_{n\geq 1}$ is a bounded and 
equicontinuous family in $\mc{C}_0(\bar\O)$ and by 
the Ascoli--Arzel\'a theorem there exists a convergent 
subsequence that we relabel in the same way $\{u_n\}$, 
such that $u_n \rightarrow u_*$ in $\mc{C}_0(\bar\O)$. 
Moreover, since the solutions of \eqref{18} are fixed points of the 
equation 
\begin{equation}
\label{68}
	u_n = (-\D-\l)^{-1}
	\left[ \g_n (-\D-\l)^{-1} u_n -af(\cdot,u_n)u_n\right],
\end{equation}
passing to the limit \eqref{68} actually shows that by the IFT curve of solutions 
can be extended beyond $\Sigma(\l)$, which is impossible. Therefore, \eqref{64} 
holds.
\par
Finally, to prove \eqref{65} we choose any increasing 
sequence $\g_n \uparrow \Sigma(\l)$. Let us set 
$w_n:=\frac{u_n}{\left\|u_n\right\|_{L^2(\O)}}$. Then, from 
\eqref{18}, after dividing the equation by the norm $\left\|u_n\right\|_{L^2(\O)}$, multiplying by $w_n$ , and then 
integrating by parts in $\O$, we find that 
$$
	\int_\O |\nabla w_n|^2 + \int_\O a(x) f(x,u_n) w_n^2 = 
	\l +\a\b \int_\O \left|(-\D-\l)^{-1/2} w_n\right|^2 \;,
$$
hence, with a similar argument as the one applied to prove the 
coercivity of the functional $\mc{E}_\g$ there exists a subsequence 
again labelled by n such that 
$$
 \lim_{n\to\infty} \|w_n-\hat{w}\|_{L^2(\O)}=0,
$$
where, $\hat{w}|_{\O_+}=0$ and $\hat{w}|_{\O_0}>0$ also in this case 
and satisfies 
$$
	\frac{\int_{\O} |\nabla \hat{w}|^2 -\l}{\int_{\O} 
	\left|(-\D-\l)^{-1/2} \hat{w}\right|^2} 
	\leq \Sigma(\l).
$$
Moreover, we claim that the equality can only hold if $\hat{w}$ has the 
form shown by \eqref{311} and belonging 
to the space $\mc{C}^2(\bar \O_0)\cup L^\infty(\O_+)$. Therefore, since 
$\hat{w}$ is strictly positive in compact sets of $\bar\O$ and owing 
to \eqref{64} we can conclude that \eqref{65} is true. This completes 
the proof.
\end{proof}

\begin{remark}
\label{Re 62}
{\rm \begin{enumerate}
\item[(i)]  Note that due to the cooperative character of the system \eqref{11},  
as mentioned above, no semi-trivial solutions are allowed. Therefore, from 
Proposition\;\ref{Pr 61} it might be also concluded that 
$(\g,u,v)=((\s_1-\l)^2,0,0)$ is a bifurcation point from 
the branch of trivial solutions and, hence, 
$\lim_{\g\downarrow (\s_1-\l)^2} v(\g)=0$ is true. 
\item[(ii)]  Furthermore, the limiting behaviour at the upper bound 
of the parameter $\g$ obtained for $u(\g)$ might be extended to 
the second component of the coexistence 
states $v(\g)$. Actually, 
according to \eqref{11}, after a straightforward calculation 
it is easily seen that 
$$
  (-\D+\sqrt{\g}-\l)\left(\sqrt{\a}\;v(\g)-\sqrt{\b}\; u(\g)\right)
  = \sqrt{\b}\; a f(\cdot,u(\g))u(\g) \qquad \hbox{in}\;\;\O.
$$
Moreover, since it has been imposed that $\l<\s_1$ we have 
$$
  \s[-\D+\sqrt{\g}-\l,\O]=\s_1-\l + \sqrt{\g}>\sqrt{\g}>0.
$$
Thus, owing to \cite[Theorem\;2.5]{Lo3}, we find that
$$
  \sqrt{\a}\;v(\g)= \sqrt{\b}\; u(\g)+
 \sqrt{\b} (-\D+\sqrt{\g}-\l)^{-1} \left( a f(\cdot,u(\g))u(\g) \right)
$$
and, therefore,
\begin{equation}
\label{e6}
  \sqrt{\a}\;v(\g) > \sqrt{\b}\; u(\g).
\end{equation}
\item[(iii)]  Finally, we would like to point out 
the strength of those cooperative systems applies to every component 
forces both components to behave in a similar way. Then a very natural extension 
of this work could be the consideration of those cooperative terms 
as functions with enough regularity instead of parameters as it has been assumed here.
\end{enumerate}
}
\end{remark}

\vspace{0.4cm}

\noindent{\bf Acknowledgements.} The author would like to express his deepest gratitude to the reviewer 
of this work for all the suggestions and corrections made to improve this work. Also, he would like 
to thank Professors Mariano Giaquinta and Pietro Majer for their help and encouragement  
in the realization of this work during the two years 
the author spent at the research centre Ennio De Giorgi--Scuola Normale Superiore of Pisa.

\end{document}